\documentclass[UTF-8,reqno,12pt]{amsart}
\usepackage{enumerate}
\usepackage{amssymb,url,color}
\usepackage{hyperref}
\usepackage{amsmath,amsthm}
\usepackage{mathrsfs}
\usepackage[notref,notcite]{showkeys}
\usepackage{bm}

\allowdisplaybreaks[4]

\newtheorem{thm}{Theorem}[section]
\theoremstyle{Condition}

\newtheorem{cor}{Corollary}[section]

\newtheorem{lem}{Lemma}[section]
\newtheorem{rem}{Remark}[section]

\theoremstyle{Problem}

\theoremstyle{Assumption}

\theoremstyle{Definition}

\numberwithin{equation}{section}

\def\beq{\begin{equation}}
\def\deq{\end{equation}}
\allowdisplaybreaks 

\bfseries\rmfamily 

\def\cC{{\mathcal C}}

\def\cF{{\mathcal F}}

\def\cP{{\mathcal P}}

\def\mE{{\mathbb E}}

\def\mN{{\mathbb N}}

\def\mR{{\mathbb R}}

\def\mW{{\mathbb W}}

\def\geq{\geqslant}
\def\leq{\leqslant}

\def\a{\alpha}
\def\om{\omega}
\def\Om{\Omega}

\def\d{\delta}

\def\l{\lambda}
\def\s{\sigma}

\def\[{{\Big[}}
\def\]{{\Big]}}
\def\<{{\langle}}
\def\>{{\rangle}}
\def\({{\Big(}}
\def\){{\Big)}}

\def\dif{{\rm d}}

\def\={&\!\!=\!\!&}
\def\bt{\begin{theorem}}
\def\et{\end{theorem}}
\def\bl{\begin{lemma}}
\def\el{\end{lemma}}
\def\br{\begin{rem}}
\def\er{\end{rem}}

\begin{document}

\title[Euler's scheme of Mckean-Vlasov SDEs with non-Lipschitz coefficients]
{Euler's scheme of Mckean-Vlasov SDEs with non-Lipschitz coefficients}

\author{Zhen Wang, Jie Ren and Yu Miao}

\address{Zhen Wang: College of Mathematics and Information Statistics, Henan Normal University, Xinxiang, Henan, 453000, P.R.China}
\email{wangzhen881025@163.com}
\address{Jie Ren: College of Mathematics and Information Statistics, Henan University of Economics and Law, Zhengzhou, Henan, 450000, P.R.China}
\email{20130006@huel.edu.cn}

\address{Yu Miao: College of Mathematics and Information Statistics, Henan Normal University, Xinxiang, Henan, 453000, P.R.China}
\email{yumiao728@gmail.com}

\thanks{This work is partially supported by an NNSFC grant of China (No.11901154) and (No. 11971154)}

\begin{abstract}
\ \ In this paper, we show the strong well-posedness of Mckean-Vlasov SDEs
 with non-Lipschitz coefficients. Moreover, propagation of chaos and the convergence rate for Euler's scheme of Mckean-Vlasov SDEs
  are also obtained.
\end{abstract}

\subjclass[2000]{}

\maketitle

\section{Introduction}

Let $\left(\Om,\cF,\cP;(\cF_t)_{t\geq 0}\right)$ be a complete filtration probability space, endowed with a standard $d$-dimensional Brownain motion $(W_t)_{t\geq 0}$ on the probability space. Consider the following Mckean-Vlasov stochastic differential equations:
\beq\label{a-1}
\dif X_{t}=\int_{\mR^d}b\left(X_{t},y\right)\mu(\dif y)\dif t+\int_{\mR^d}\s\left(X_{t},y\right)\mu(\dif y)\dif W_t,\ \ X_0=\xi,
\deq
where $\mu$ be a probability measure on $\mR^d$, the initial value $X_0=\xi$ is a $\nu$-distributed, and $\cF_0$-measurable $\mR^d$-valued random variable  and the coefficients
$$
b:\mR^d\times \mR^d\rightarrow\mR^d,\ \  \s:\mR^d\times \mR^d\rightarrow\mR^d\otimes\mR^d
$$
are Borel measurable functions.

Mckean-Vlasov SDEs were  initiated by Mckean \cite{M} using Kac's formalism of molecular chaos \cite{K}, and have attracted  substantial research interests(cf.~\cite{S}). R$\mathrm{\ddot{o}}$ckner and Zhang \cite{RZh} showed the strong well-posedness of the above SDE under a singular distribution dependent drift and a constant matrix $\s$. Recently, Mishura and Veretennikov \cite{MV} established weak and strong existence and uniqueness results for multi-dimensional stochastic McKean-Vlasov equations if  $b$ and $\s$ are of linear growth in $x$ and the diffusion matrix $\s$ is uniformly non-degenerate. Up to now, there are many works devoted to the study of Mckean-Vlasov SDEs(see \cite{d,HW1,LM} and references therein). Moreover, the non-degenerate assumption on $\s$ is usually required when $\s$ is non-constant.

We note that if
$$
\bar{b}(x,\mu):=\int_{\mR^d}b(x,y)\mu(\dif y),\ \ \bar{\s}(x,\mu):=\int_{\mR^d}\s(x,y)\mu(\dif y),
$$
then Eq.(\ref{a-1}) can be rewritten as
\beq\label{1-1}
\dif X_{t}=\bar{b}\left(X_{t},\mu_{t}\right)\dif t+\bar{\s}\left(X_{t},\mu_{t}\right)\dif W_t,\ \ X_0=\xi,
\deq
where $\mu_{t}$ is the law of $X_t$, and
$$
\bar{b}:\mR^d\times \cP_p(\mR^d)\rightarrow\mR^d, \ \ \bar{\sigma}:\mR^d\times \cP_p(\mR^d)\rightarrow\mR^d\otimes\mR^d
$$
are Borel measurable functions. Here let $\cP_p(\mR^d)$ denote a space of probability measures on $\mR^d$ with $p$-th moment($p\geq 1$), that is, $\mu(|\cdot|^p):=\int_{\mR^d}|x|^p\mu({\dif x})<\infty$. What's more, Eq.(\ref{1-1}) is also called mean-field SDE or distribution dependent SDE in the literature, which naturally appears in the studies of interacting systems and mean-field games (e.g. \cite{CD,CD1,CGPS,K,M,S,V} ).

 Meanwhile, numerous results are developed under irregular
 conditions, based on Eq.(\ref{1-1}). Wang \cite{W} studied  strong well-posedness of   distribution dependent SDEs with one-sided Lipschitz continuous drift $\bar{b}$ and Lipschitz-continuous diffusion $\bar{\s}$.
For  additive noise case, existence and uniqueness of strong solutions with irregular drift were established in \cite{BMP}. Huang and Wang \cite{HW} proved the existence and uniqueness of distribution dependent SDEs with non-degenerate noise, under integrability conditions on distribution dependent coefficients. Ding and Qiao \cite{DQ}  obtained the strong well-posedness  under non-Lipschitz coefficients  by weak existence and pathwise
uniqueness.
In \cite{MBKM}, the unique strong
solution was constructed by effective approximation procedures, without using the famous
Yamada-Watanabe theorem.  So far, related topics for Eq.(\ref{1-1}) have been considerably investigated, such as Harnack inequality \cite{HW,W}, ergodicity \cite{EGZ,RRW}, and Feynman-Kac formulae \cite{BLPR,CM,RRW}  and so on.
In general, the model of (\ref{a-1}) has comparatively more widespread applications than (\ref{1-1}), which don't require the coefficients depend on
the law of the solution, such as $b(x-y)$.

In the present paper, we make the following non-Lipschitz assumptions without non-degeneracy diffusion:\\
$\mathbf{(H_1)}$ For any $x,y\in\mR^d$,
and for some $c_0>0$,
$$
|b(x,y)|+\|\s(x,y)\|\leq c_0(1+|x|+|y|).
$$
$\mathbf{(H_2)}$ The functions $b,\s$ are Borel continuous measurable functions in $\mR^d\times\mR^d$. For all $ x,y\in\mR^d$, there are constants $\l_1,\l_2>0$ such that
$$
|b(x_1,y_1)-b(x_2,y_2)|\leq \l_1\left(|x_1-x_2|\gamma_1(|x_1-x_2|)+|y_1-y_2|\gamma_1(|y_1-y_2|)\right),
$$
$$
\|\s(x_1,y_1)-\s(x_2,y_2)\|^2\leq \l_2\left(|x_1-x_2|^2\gamma_2(|x_1-x_2|)+|y_1-y_2|^2\gamma_2(|y_1-y_2|)\right),
$$
where $\gamma_i$ is a positive continuous  function on $\mR^+$, bounded  on $[1,\infty)$ and satisfying
\beq\label{1-7}
\lim_{x\downarrow0}\frac{\gamma_i(x)}{\log(x^{-1})}=\d_i<\infty, i=1,2.
\deq

Throughout this paper, we assume $\mE|\xi|^2<\infty$ and all relevant stochastic differential equations admit unique solutions.
One of the main results of this paper is stated as follows.

\begin{thm}\label{thm-a}Assume $\mathbf{(H_1)}$ and $\mathbf{(H_2)}$. For any $T>0$, (\ref{a-1}) has a unique strong solution $(X_t)_{t\geq0}$ such that
\beq\label{1-a}
\mE\left[\sup_{t\in[0,T]}|X_t|^2\right]\leq C(1+\mE|\xi|^2),
\deq
where $C$ is a constant depending on $T$ and $c_0$.
\end{thm}
\begin{rem}In this paper, we state the existence of solutions to (\ref{a-1}) using the standard
method-Picard successive approximation. Meanwhile, there are other ways to show the existence of solutions to (\ref{a-1}), by the well-known result for bounded measurable drift $b$ obtained in \cite{MV} and the estimate of uniformity in \cite{XXZZ}, which is established by Krylov's estimate and Zvonkin's technique, and whose way is obtained under non-degenerate coefficient.
Thereby, for degenerate diffusion terms, the strong well-posedness of Mckean-Vlasov SDEs is established with non-Lipschitz coefficients, in this article.
\end{rem}
\begin{rem}
We also mention our conditions. Theorem~\ref{thm-a} can not be covered by  the recent work of \cite{DQ} in which the distribution dependent coefficients are non-Lipschitz continuous in the spatial variables  and  Lipschitz continuous in law under the Wasserstein metric.   However,  in this paper, if $b(x,y)$ and $\s(x,y)$ are Lipschitz continuous in $y$, it implies the conditions of the corresponding distribution dependent coefficients in \cite{DQ}. 
  That is, the  conclusion improves  the result of \cite[Theorem~3.1]{DQ}.
\end{rem}

Suppose that the sequence of $\{\xi_i,i\in\mN\}$ is independent identically distributed with a common distribution $\nu$ in $\mR^d$ and $\{W^i,i\in\mN\}$ is a sequence of independent $d$-dimensional Brownian motions. Consider
the following non-interacting particle systems:
\beq\label{1-3}
\dif X_{t}^{i}=\int_{\mR^d}b(X_{t}^{i},y)\mu_{X_{t}^{i}}(\dif y)\dif t+\int_{\mR^d}\s(X_{t}^{i},y)\mu_{X_{t}^{i}}(\dif y)\dif W_t^i, \ \ X_{0}^i=\xi_i,
\deq
where $\mu_{X_{t}^i}$ stands for the law of $X_{t}^i$, and $\{X_{\cdot}^i,i\in\mN\}$ is a set of independent identically distributed stochastic processes with the common distribution $\mu_{\cdot}$.

Next, We use the  discretized $N$-interacting particle to approximate the above non-interacting particle. Namely, consider $N$-interacting particle $X^{N,i}$ satisfying
\beq\label{a-9}
\dif X_{t}^{N,i}=\bar{b}\left(X_{t}^{N,i},\mu_{t}^N\right)\dif t+\bar{\s}\left(X_{t}^{N,i},\mu_{t}^N\right)\dif W_t^i,\ \ X_0^{N,i}=\xi_i,
\deq
where $\mu_t^N$ is the empirical measure of $\left\{X_t^{N,i},i=1,2,\cdots,N\right\}$ defined by
$$
\mu_t^N:=\frac{1}{N}\sum_{j=1}^N\d_{X_t^{N,j}},
$$
where $\d_x$ stands for a Dirac measure at point $x$.
Thus, for $i=1,2,\cdots, N$ and for any $t\geq 0$, (\ref{a-9}) is equivalent to
\beq\label{1-2}
\dif X_t^{N,i}=\frac{1}{N}\sum_{j=1}^Nb\left(X_{t}^{N,i},X_{t}^{N,j}\right)\dif t+\frac{1}{N}
\sum_{j=1}^N\s\left(X_{t}^{N,i},X_{t}^{N,j}\right)\dif W_t^i,\ \ X_0^{N,i}=\xi_i.
\deq

The so-called propagation of chaos is that the
law of fixed particles $X_t^{N,i}$ tends to the distribution of   independent particles $X_{t}^{i}$
solving (\ref{1-3}) with same law when  $N$  goes to $+\infty$ .
The study of chaos have already profound impacts on theoretical and applied contexts. In \cite{JW}, there are more developments with an eye toward models arising in the physical sciences. By the mean field game theory, \cite{CST,HMC} have inspired new theoretical developments and applications in engineering and economics.

The propagation of chaos for the interacting particle of Mckean-Vlasov SDEs has been revealed under Lipschitz assumptions, see \cite{S}. We mention that new particle representations  for ergodic McKean-Vlasov SDEs were introduced in \cite{ABRS}. Propagation of chaos and convergence rate of Euler's approximation were established by Bao and Huang \cite{BH}, where the coefficients are H\"older continuous in $x$ and Lipschitz continuous in $\mu$.
Especially in the case of nondegenerate  diffusion terms, Zhang \cite{Zh2} also showed the propagation of chaos for Euler's approximation with the linear growth coefficients. Rencently, Dos Reis, Smith and Tankov \cite{dST} took some examples in many different fields with Lipschitz coefficients.

The following results state the propagation of chaos for stochastic interacting particle systems. In other words,
the continuous time Euler's scheme of stochastic $N$-interacting particle systems converges strongly to non-interacting particle systems.
\begin{thm}\label{thm-1}Suppose that $\mathbf{(H_1)}$ and  $\mathbf{(H_2)}$ hold. Then for any $T>0$,
\beq\label{1-4}
\lim_{N\rightarrow\infty}\sup_{i=1,2,\cdots,N}\mE\left(\sup_{ t\in[0,T]}\left|X_t^{N,i}-X_t^i\right|^2\right)=0.
\deq
In particular,
$$
\sup_{i=1,2,\cdots,N}N\mE\left[\sup_{t\leq T}\left|X_t^{N,i}-X_t^i\right|^2\right]<\infty.
$$
\end{thm}

Another goal of this paper is to get the corresponding overall convergence rate. To discretize (\ref{1-2}) in time, for fixed $h\in(0,1)$, the corresponding Euler's scheme is
\beq\label{1-2'}
\dif X_t^{h,N,i}=\frac{1}{N}\sum_{j=1}^Nb\left(X_{t_h}^{h,N,i},X_{t_h}^{h,N,j}\right)\dif t+\frac{1}{N}
\sum_{j=1}^N\s\left(X_{t_h}^{h,N,i},X_{t_h}^{h,N,j}\right)\dif W_t^i,\ \ X_0^{h,N,i}=\xi_i.
\deq

\begin{thm}\label{thm-2}Under the assumptions of Theorem \ref{thm-1}, for $h\in(0,1)$ sufficiently small and $\a\in(0,\frac{1}{2})$, there exist a constant $C$ independent of $h$ and $N$ such that
\beq\label{1-4}
\sup_{i=1,2,\cdots,N}\mE\left(\sup_{ t\in[0,T]}\left|X_t^{h,N,i}-X_t^i\right|^2\right)< C\left(\frac{1}{N}+h^{2\alpha}\right).
\deq

\end{thm}

This paper is organized as follows. In Section 2, we show the main lemmas  for later use. Section 3 yields the strong wellposedness of (\ref{a-1}). In section 4, we construct the propagation of chaos and  the corresponding overall convergence rate.

Throughout the paper, $C$ with or without
indices will denote different positive constants (depending on the indices), whose values may change from one place to another and not important.
\section{Preliminary}
First, We introduce a well-studied metric on the space of
distributions known as the Wasserstein distance which allows us to consider $\cP_p(\mR^d)$ as a metric space.

For $\mu,\nu\in\cP_p(\mR^d)$, the $\mW_p$-Wasserstein distance between $\mu$ and $\nu$ is defined by
\beq\label{2-a}
\mW_p(\mu,\nu)=\inf_{\pi\in\cC(\mu,\nu)}\left(\int_{\mR^d\times \mR^d}|x-y|^p\pi(\dif x,\dif y)\right)^{\frac{1}{1\vee p}},
\deq
where $\cC(\mu,\nu)$ is the set of all couplings of $\mu$ and $\nu$ on $\cP(\mR^d\times \mR^d)$, such that $\pi\in\cC(\mu,\nu)$ if and only if $\pi(\cdot,\mR^d)=\mu(\cdot)$ and $\pi(\mR^d,\cdot)=\nu(\cdot)$.

The topology induced by Wasserstein metric  coincides with the topology of weak convergence of measure together with
the convergence of all moments of order up to $p$ in \cite{D}.

For later use, for  $0<\eta<1/e$, we define a  strictly increasing, and concave function:
$$
\rho_{\eta}(x)=
\left\{
\begin{array}{ll}
x\log x^{-1},~~~&0<x\leq\eta, \\
(\log \eta^{-1}-1)x+\eta,~~~&x>\eta.\\
\end{array}
\right.
$$

The generalization of Gronwall-Bellman type inequality has been verified by Bihari \cite{B}. It can be  also found in \cite{Zh1}.
\begin{lem}\label{lem-1} Let $g(s), q(s)$ be two strictly positive functions on $\mR_+$ satisfying $g(0)<\eta$ and
$$
g(t)\leq g(0)+\int_0^tq(s)\rho_{\eta}(g(s))\dif s,\ \ t\geq 0.
$$
Then
$$
g(t)\leq g(0)^{exp\{-\int_0^tq(s)\dif s\}}.
$$
\end{lem}
\begin{lem}\label{lem-2}Assume $\mathbf{(H_1)}$ and $\mathbf{(H_2)}$. Then there exists a constant $C$  independent of $N$, such that for all $T>0$, $N\in\mN$ and $i=1,2,3,\cdots,N$,
\beq\label{2-3}
\mE\Big[\sup_{t\in[0,T]}\big|X_t^{N,i}\big|^2\Big]\leq C(1+\mE|\xi_i|^2).
\deq
\end{lem}
\begin{proof}By It\^o's formula, H\"older's inequality and the linear growth of $b$ and $\s$, we have
$$
\aligned
&\mE\Big[\sup_{t\in[0,T]}\big|X_t^{N,i}\big|^2\Big]=\\
\leq& \mE|\xi_i|^2 + C\int_0^T\mE\Big[\sup_{r\in[0,s]}\Big(\left|X_r^{N,i}\right|^2+\frac{1}{N}\sum_{j=1}^N
\left|X_r^{N,j}\right|^2+1\Big)\Big]\dif s,\\
\endaligned
$$
where Jensen's inequality and the condition $\mathbf{(H_1)}$ are used in the last inequality.

Summing over $i$ on both sides, we have
$$
\aligned
&\sum_{i=1}^N\mE\Big[\sup_{t\in[0,T]}\big|X_t^{N,i}\big|^2\Big]\\
\leq&\sum_{i=1}^N \mE|\xi_i|^2+ C\int_0^T\Bigg[\mE\Big(\sup_{r\in[0,s]}\sum_{i=1}^N\left|X_r^{N,i}\right|^2\Big)\\
&+\mE\bigg(\sup_{r\in[0,s]}N\cdot\Big[\frac{1}{N}\sum_{j=1}^N
\left|X_r^{N,j}\right|^2+1\Big]\bigg)\Bigg]\dif s.\\
\endaligned
$$
By the symmetry, it holds that
$$
\sum_{i=1}^N\mE\Big[\sup_{t\in[0,T]}\big|X_t^{N,i}\big|^2\Big]\leq \sum_{i=1}^N\mE|\xi_i|^2 + C\int_0^T\mE\Big[\sup_{r\in[0,s]}\sum_{i=1}^N\left(2\left|X_r^{N,i}\right|^2+1\right)\Big]\dif s.
$$
Using symmetry, and removing the  summation symbol, then
$$
\mE\Big[\sup_{t\in[0,T]}\big|X_t^{N,i}\big|^2\Big]\leq \mE|\xi_i|^2 + C\int_0^T\mE\Big[\sup_{r\in[0,s]}\left(2\left|X_r^{N,i}\right|^2+1\right)\Big]\dif s.
$$
With the aid of Gronwall's inequality, we can obtain
$$
\mE\Big[\sup_{t\in[0,T]}\big|X_t^{N,i}\big|^2\Big]\leq C(1+\mE|\xi_i|^2),
$$
where the constant $C$ depends on $T, c_0$ but independent of $N$.
\end{proof}
\begin{cor}\label{cor-1}Under conditions $\mathbf{(H_1)}$ and $\mathbf{(H_2)}$, for all $T>0$, $N\in\mN$ and $i=1,2,3,\cdots,N$, there exists a constant $C$  independent of $N$ such that
\beq\label{2-2}
\mE\Big[\sup_{t\in[0,T]}\big|X_t^{h,N,i}\big|^2\Big]\leq C(1+\mE|\xi_i|^2).
\deq
\end{cor}
\begin{proof} It is similar to the proof of Lemma \ref{lem-2} about main verified processes. Combining with
$$
\sup_{0\leq r\leq s}\big|X_{r_h}^{h,N,i}\big|^2\leq \left|X_s^{h,N,i}\right|^2,
$$
the conclusion is true.
\end{proof}
\begin{lem}\label{lem-3} Under the hypothesis of Lemma \ref{lem-2}, we have for any $s,t\in[0,T]$,
\beq\label{lem-3-1}
\mE\left|X_t^{N,i}-X_s^{N,i}\right|^2\leq C(1+\mE|\xi_i|^2)(t-s)
\deq
and
\beq\label{lem-3-2}
\mE\left|X_t^{h,N,i}-X_s^{h,N,i}\right|^2\leq C(1+\mE|\xi_i|^2)(t-s),
\deq
where $C$ is a constant depending on $T$ and $c_0$ but independent of $N$.
\end{lem}
\begin{proof}
By the H\"older's inequality, BDG's inequality and the assumption $\mathbf{(H_1)}$, it holds that
$$
\aligned
&\mE\left|X_t^{N,i}-X_s^{N,i}\right|^2=
\\ \leq & 2\mE\bigg[\Big|\int_s^t\frac{1}{N}\sum_{j=1}^Nb\left(X_{u}^{N,i},X_{u}^{N,j}\right)\dif u\Big|^2+\Big|\int_s^t\frac{1}{N}
\sum_{j=1}^N\s\left(X_{u}^{N,i},X_{u}^{N,j}\right)\dif W_u^i\Big|^2\bigg]\\
\leq&2\int_s^t\mE\Big[\frac{1}{N}\sum_{j=1}^N\left(\left|b\left(X_{u}^{N,i},
X_{u}^{N,j}\right)
\right|^2+\left|\s\left(X_{u}^{N,i},X_{u}^{N,j}\right)\right|^2\right)\Big]\dif u \\
\leq& 2c_0\int_s^t\mE\Big[\frac{1}{N}\sum_{j=1}^N\left(\left|X_{u}^{N,i}\right|^2
+\left|X_{u}^{N,j}\right|^2
+1\right)\Big]\dif u.\\
\endaligned
$$
Again summing and removing over $i$ of both sides with using symmetry of above inequality, then
$$
\mE\left|X_t^{N,i}-X_s^{N,i}\right|^2\leq C\int_s^t\mE\left(\left|X_{u}^{N,i}
\right|^2+1\right)\dif u,
$$
where the constant $C$ depends on $c_0$.

At last, by Gronwall's inequality, we conclude the desired estimate. In the similar way as above  discussions, it holds that $X_t^{h,N,i}$ satisfies Equation (\ref{lem-3-2}).
\end{proof}
\section{Proof of Theorem \ref{thm-a}}
Based on the conditions $\mathbf{(H_1)}$ and $\mathbf{(H_2)}$ for Eq.(\ref{a-1}), we have
$$
|\bar{b}(x,\mu)|+|\bar{\s}(x,\mu)|\leq C\left(1+|x|+\mW_1(\mu,\delta_0)\right).
$$
Using (\ref{1-7}) and the definition of $\rho_{\eta}$, there is a sufficiently small $0<\eta<1/e$, for any $x\in\mR^+,$ such that
  \beq\label{a-2}
 x\gamma_i(x)\leq \rho_{\eta}(x),
 \deq
 and
 \beq\label{a-4}
x^2\gamma_i(x)\leq \rho_{\eta}(x^2).
\deq
By $\mathbf{(H_1)}$ and $\mathbf{(H_2)}$ and Jensen's inequality, similarly as in the proof of \cite[Lemma~1.3]{A}, we have
$$
\aligned
|\bar{b}(x_1,\mu_1)-\bar{b}(x_2,\mu_2)|
=&\Big|\int_{\mR^d}b(x_1,y_1)\mu_1(\dif y_1)-\int_{\mR^d}b(x_1,y_2)\mu_1(\dif y_1)\Big|\\
\quad\quad&+\Big|\int_{\mR^d}b(x_1,y_2)\mu_1(\dif y_1)-\int_{\mR^d}b(x_2,y_2)\mu_2(\dif y_2)\Big|\\
&\leq\lambda_1|x_1-x_2|\gamma_1(|x_1-x_2|)+\lambda_1\int_{\mR^d\times \mR^d}\rho_{\eta}(|y_1-y_2|)\pi(\dif y_1,\dif y_2)\\
&\leq \lambda_1|x_1-x_2|\gamma_1(|x_1-x_2|)+\lambda_1\rho_{\eta}\big(\mW_1(\mu_1,\mu_2)\big).
 \endaligned
$$
By the similar deduction to above it holds that
$$
\|\bar{\s}(x_1,\mu_1)-\bar{\s}(x_2,\mu_2)\|^2\leq \lambda_2|x_1-x_2|^2\gamma_2(|x_1-x_2|)+\lambda_2\rho_{\eta}\big(\mW_2(\mu_1,\mu_2)^2\big).
$$

Especially, if $b(x,y)$ and $\s(x,y)$ are Lipschitz continuous in $y$, the corresponding $\bar{b}(x,\mu)$ and $\bar{\s}(x,\mu)$  are Lipschitz continuous in $\mu$. In this situation, by \cite[Theorem~3.1]{DQ},  the weak existence of (\ref{1-1}) and pathwise uniqueness imply the strong existence of (\ref{a-1}).
In this part, we will
directly construct the strong solutions to (\ref{a-1}) by the method of successive approximations under  the more general conditions $\mathbf{(H_1)}$ and $\mathbf{(H_2)}$.

\begin{proof}[Proof of Theorem \ref{thm-a}]

 Consider the following distribution-iterated SDEs: for any $t\in[0,T]$ and for each $k\geq 1$,
\beq\label{1-8}
\aligned
 X_t^{(k)}=&X_0^{(k)}+\int_0^t\int_{\mR^d}b(X_s^{(k)},\om_s^{(k-1)})\mu_{k-1}(\dif \om^{(k-1)})\dif s\\ &+\int_0^t\int_{\mR^d}\s(X_s^{(k)},\om_s^{(k-1)})\mu_{k-1}(\dif \om^{(k-1)})\dif W_s,\\
 \endaligned
\deq
where $X_0^{(k)}=\xi$ and $\mu_{k-1}$ stands for the law of $X_s^{(k-1)}$.

By \cite[Theorem~4.1]{Zh1}, Eq.(\ref{1-8})  has a unique solution $(X_t^{(k)})_{t\geq 0}$. Next, existence of solutions to (\ref{a-1}) can be shown by verifying the Cauchy sequence.

With BDG's and H\"older's inequality, it follows that
\beq\label{a-6}
\mE\Big[\sup_{t\in[0,T]}\big|X_t^{(1)}\big|^2\Big]\leq C(1+\mE|\xi|^2),
\deq
where the detailed proof can be found  in literature \cite{BH} or \cite{DQ}.

Indeed, this can be handled in same manner by using the triple $(X^{(k+1)},X^{(k)},\mu^{(k)})$ in lieu of
$(X^{(1)},\xi,\nu)$. Therefore, (\ref{a-6}) still holds true for $k+1$. That is, for each $k\geq 1$, it holds that
\beq\label{aa-6}
\mE\Big[\sup_{t\in[0,T]}\big|X_t^{(k)}\big|^2\Big]\leq C(1+\mE|\xi|^2).
\deq

For notation brevity, we set
$$
Z_t^{(k+1)}:=X_t^{(k+1)}-X_t^{(k)}.
$$
By It\^{o}'s formula, then
$$
\aligned
|Z_t^{(k+1)}|^2&=\int_0^t\int_{\mR^d}Z_s^{(k+1)}\Big[b(X_s^{(k+1)},\om_s^{(k)})\mu_k(\dif \om^{(k)})-b(X_s^{(k)},\om_s^{(k-1)})\mu_{k-1}(\dif \om^{(k-1)})\Big]\dif s\\
&\quad+\int_0^t\int_{\mR^d}Z_s^{(k+1)}\left[\s(X_s^{(k+1)},\om_s^{(k)})\mu_k(\dif \om^{(k)})-\s(X_s^{(k)},\om_s^{(k-1)})\mu_{k-1}(\dif \om^{(k-1)})\right]\dif W_s.\\
&\quad+\frac{1}{2}\int_0^t\int_{\mR^d}\left\|\s(X_s^{(k+1)},\om_s^{(k)})\mu_k(\dif \om^{(k)})-\s(X_s^{(k)},\om_s^{(k-1)})\mu_{k-1}(\dif \om^{(k-1)})\right\|^2\dif s\\
&=:I_1(t)+I_2(t)+I_3(t).
\endaligned
$$
By $\mathbf{(H_1)}$, we show
$$
\aligned
I_1(t)=
&\int_0^t\int_{\mR^d}Z_s^{(k+1)}\Big[b(X_s^{(k+1)},\om_s^{(k)})\mu_k(\dif \om^{(k)})-b(X_s^{(k+1)},\om_s^{(k-1)})\mu_{k-1}(\dif \om^{(k-1)})\Big]\dif s\\
&+\int_0^t\int_{\mR^d}Z_s^{(k+1)}\Big[b(X_s^{(k+1)},\om_s^{(k-1)})\mu_{k-1}(\dif \om^{(k-1)})-b(X_s^{(k)},\om_s^{(k-1)})\mu_{k-1}(\dif \om^{(k-1)})\Big]\dif s\\
\leq&C\int_0^t\Big[|Z_s^{(k+1)}|^2\gamma_1(|Z_s^{(k+1)}|)+Z_s^{(k+1)}\rho_{\eta}\Big(\mW_1(\mu_s^{(k)},\mu_s^{(k-1)})\Big)\Big]\dif s\\
\endaligned
$$
Similarly,
by $\mathbf{(H_2)}$, we have
$$\aligned
I_3(t)\leq &C\int_0^t\Big[|Z_s^{(k+1)}|^2\gamma_2(|Z_s^{(k+1)}|)\Big.\\
&\quad+\Big.\int_{\mR^d\times \mR^d}|\om_s^{(k)}-\om_s^{(k-1)}|^2\gamma_2(|\om_s^{(k)}-\om_s^{(k-1)}|) \pi(\dif\om_s^{(k)},\dif\om_s^{(k-1)})\Big]\dif s\\
\leq&C\int_0^t\Big[\rho_\eta(|Z_s^{(k+1)}|^2)+\rho_{\eta}\Big(\mW_2\left(\mu_s^{(k)},\mu_s^{(k-1)}\right)^2\Big)\Big]\dif s
\endaligned
$$
By  BDG's inequality and Young's inequality,  we have $$\aligned
\mE\bigg(\sup_{t\in[0,T]} I_2(t)\bigg)&\leq C\mE\bigg(\int_0^T\int_{\mR^d}\Big[|Z_s^{(k+1)}|^2\rho_\eta(|Z_s^{(k+1)}|)+|Z_s^{(k+1)}|^2\rho_{\eta}\Big(\mW_2\left(\mu_s^{(k)},\mu_s^{(k-1)}\right)^2\Big)\Big]\dif s\bigg)^{\frac{1}{2}}\\
&\leq C\mE\bigg(\sup_{t\in[0,T]}|Z_t^{(k+1)}|^2\int_0^T\int_{\mR^d}\Big[\rho_\eta(|Z_s^{(k+1)}|)+\rho_{\eta}\Big(\mW_2\left(\mu_s^{(k)},\mu_s^{(k-1)}\right)^2\Big)\Big]\dif s\bigg)^{\frac{1}{2}}\\
&\leq C\mE\bigg(\sup_{t\in[0,T]}\int_0^t\int_{\mR^d}\Big[\rho_\eta(|Z_s^{(k+1)}|)+\rho_{\eta}\Big(\mW_2\left(\mu_s^{(k)},\mu_s^{(k-1)}\right)^2\Big)\Big]\dif s\bigg)\\
&\quad+ \frac{1}{4}\mE\bigg(\sup_{t\in[0,T]} |Z_t^{(k+1)}|^2\bigg)\\
\endaligned
$$
Then, with Jensen's inequality and Young's inequality,
$$
\aligned
\mE\bigg(\sup_{t\in[0,T]} |Z_t^{(k+1)}|^2\bigg)&\leq C\int_0^T\rho_{\eta}\bigg(\mE\Big(\sup_{r\in[0,s]}|Z_r^{(k+1)}|^2\Big)\bigg)\dif r\\
&\quad+C\int_0^T\bigg[\rho^2_{\eta}\Big(\sup_{r\in[0,s]}\mW_1(\mu_r^{(k)},\mu_r^{(k-1)})\Big)+\rho_{\eta}\Big(\sup_{r\in[0,s]}\mW_2(\mu_r^{(k)},\mu_r^{(k-1)})^2\Big)\bigg]\dif r\\
&\leq C\int_0^T\rho_{\eta}\bigg(\mE\Big(\sup_{r\in[0,s]}|Z_r^{(k+1)}|^2\Big)\bigg)\dif r\\
&\quad+C\int_0^T\bigg[\rho^2_{\eta}\Big(\mE\Big(\sup_{r\in[0,s]}|Z_r^{(k)}|^2\Big)^{\frac{1}{2}}\Big)+\rho_{\eta}\Big(\mE\Big(\sup_{r\in[0,s]}|Z_r^{(k)}|^2\Big)\Big)\bigg]\dif r,
\endaligned
$$
where the last inequality is established with properties of $\mW_p$-Wasserstein distance, as following:
\beq
\sup_{r\in[0,s]}\mW_1\left(\mu_r^{(k)},\mu_r^{(k-1)}\right)\leq \mE\Big(\sup_{r\in[0,s]}\left|X_r^{(k)}-X_r^{(k-1)}\right|^2\Big)^{\frac{1}{2}},
\deq
and
\beq
\sup_{r\in[0,s]}\mW_2\left(\mu_r^{(k)},\mu_r^{(k-1)}\right)^2\leq \mE\Big(\sup_{r\in[0,s]}\left|X_r^{(k)}-X_r^{(k-1)}\right|^2\Big).
\deq

Furthermore, with the aid of  (\ref{a-6}), 
there exists a constant $a$ such that
$$\mE\Big[\sup_{r\in[0,s]}\left|Z_r^{(1)}\right|^2\Big]\leq a.$$ Therefore,
by iteration and  Lemma \ref{lem-1}, we have
\beq\label{a-5}
\mE\Big[\sup_{t\in[0,T]}\big|Z_t^{(k+1)}\big|^2\Big]\leq [\rho^2_\eta(\sqrt{a})+\rho_{\eta}(a)]^{\exp(-kCT)}.
\deq

Then, there is an $(\cF_t)_{t\in[0,T]}$-adapted continuous stochastic process $(X_t)_{t\in[0,T]}$ and $\mu_t$ is the law of $X_t$ such that
\beq\label{aaa-5}
\lim_{k\rightarrow\infty}\mE\Big[\sup_{t\in[0,T]}\big|X_t^{(k+1)}-X_t\big|^2\Big]=0,
\deq
and
\beq\label{ac-5}
\lim_{k\rightarrow\infty}\sup_{t\in[0,T]}\mW_2\left(\mu_t^{(k)},\mu_r\right)^2\leq \lim_{k\rightarrow\infty}\mE\bigg[\sup_{t\in[0,T]}\big|X_t^{(k)}-X_t\big|^2\bigg]=0.
\deq
By $\mathbf{(H_1)}$, we find
$$
\aligned
&\int_0^t\int_{\mR^d}\left[b(X_s^{(k+1)},\om_s^{(k)})\mu_k(\dif \om^{(k)})-b(X_s,\om_s)\mu(\dif \om)\right]\dif s\\
&\leq C\int_0^T\left[\int_{\mR^d\times \mR^d}|\om_s^{(k)}-\om_s|\gamma_1\big(|\om_s^{(k)}-\om_s|\big) \pi(\dif\om_s^{(k)},\dif\om_s)\right]\dif s\\
\quad\quad&+ C\int_0^T\left[|X_s^{(k+1)}-X_s|\gamma_1\big(|X_s^{(k+1)}-X_s|\big) \right]\dif s\\
&\leq C\int_0^T\left[
\rho_{\eta}\big(\mW_1(\mu_s^{(k)},\mu_s)\big)+\rho_{\eta}\big(|X_s^{(k+1)}-X_s|\big)\right]\dif s
\endaligned
$$
By  dominated
convergence theorem and (\ref{aaa-5}) and (\ref{ac-5}), we  obtain
$$\lim_{k\rightarrow\infty}\mE\left|\int_0^T\int_{\mR^d}\left[b(X_s^{(k+1)},\om_s^{(k)})\mu_k(\dif \om^{(k)})-b(X_s,\om_s)\mu(\dif \om)\right]\dif s\right|=0.$$

Similarly,
$$\lim_{k\rightarrow\infty}\mE\left|\int_0^T\int_{\mR^d}\left[\s(X_s^{(k+1)},\om_s^{(k)})\mu_k(\dif \om^{(k)})-\s(X_s,\om_s)\mu(\dif \om)\right]\dif W_s\right|=0.$$
By taking $k\rightarrow\infty$ in the equation (\ref{1-8}), we derive  (by extracting a suitable subsequence)
SDE (\ref{a-1}).

On the other hand, we show the uniqueness of (\ref{a-1}). We assume that $(X_t^1)_{t\geq 0}$ and $(X_t^2)_{t\geq 0}$ are solutions to (\ref{a-1}) with the same initial value $\xi$. By the similar way, one has
$$
\lim_{k\rightarrow\infty}\mE\Big[\sup_{t\in[0,T]}\left|X_t^{1}-X_t^{2}\right|^2\Big]\leq C\int_0^T\rho_{\eta}\Big(\mE\Big[\sup_{r\in[0,s]}|X_r^{1}-X_r^{2}|^2\Big]\Big)\dif s.
$$
Once again invoking Lemma \ref{lem-1}, yields the uniqueness.

Finally, we intend to show (\ref{1-a}). Applying H\"older's inequality, Jensen's inequality, BDG's inequality and the condition $\mathbf{(H_1)}$, we show
$$
\aligned
\mE\Big[\sup_{t\in[0,T]}\left|X_t\right|^2\Big]\leq & 3\bigg\{\mE|\xi|^2+T\int_0^T\mE\Big[\sup_{r\in[0,s]}\Big|\int_{\mR^d}b(X_r,y)\mu(\dif y)\Big|^2\Big]\dif s .\\
&+\int_0^T\mE\Big[\sup_{r\in[0,s]}\Big|\int_{\mR^d}\s(X_r,y)\mu(\dif y)\Big|^2\Big]\dif s\bigg\} \\
\leq&C\bigg\{1+\mE|\xi|^2+\int_0^T\mE\Big[\sup_{r\in[0,s]}\left|X_r\right|^2\Big]\dif s+\int_0^T\Big[\sup_{r\in[0,s]}\mW_1(\mu_r,\d_0)^2\dif s\Big]\bigg\}\\
\leq &C(1+\mE|\xi|^2)+C\int_0^T\mE\Big(\sup_{r\in[0,s]}\left|X_r\right|^2\Big)\dif s .\\
\endaligned
$$
Together with Gronwall's inequality, implies
$$
\mE\Big[\sup_{t\in[0,T]}|X_t|\Big]\leq C(1+\mE|\xi|^2).
$$
So the proof is finished.
\end{proof}
\section{Proof of Theorem \ref{thm-1}}
\begin{proof}[Proof of Theorem \ref{thm-1}]Under the conditions $\mathbf{(H_1)}$ and $\mathbf{(H_2)}$, combining with Eq.(\ref{1-2}) and (\ref{1-3}), we find that
\begin{align*}
X_t^{N,i}-X_t^i=&\int_0^t\Big[\frac{1}{N}\sum_{j=1}^Nb(X_{s}^{N,i},
X_{s}^{N,j})
-\int_{\mR^d}b(X_{s}^i,y)\mu_{X_{s}^i}(\dif y)\Big]\dif s\\
&+\int_0^t\Big[\frac{1}{N}
\sum_{j=1}^N\s\left(X_{s}^{N,i},X_{s}^{N,j}\right) -\int_{\mR^d}\s(X_{s}^i,y)\mu_{X_{s}^i}(\dif y)\Big]\dif W_s^i\\
=&\int_0^t\frac{1}{N}\sum_{j=1}^N\Big[\{b(X_{s}^{N,i},X_{s}^{N,j})-b(
X_{s}^{N,i},X_{s}^{N,j})+\tilde{b}(
X_{s}^{N,i},X_{s}^{N,j})\Big] \dif s\\
&+\int_0^t\frac{1}{N}\sum_{j=1}^N\Big[\s(X_{s}^{N,i},X_{s}^{N,j})-
\s(
X_{s}^{N,i},X_{s}^{N,j})+\tilde{\s}\left(
X_{s}^{N,i},X_{s}^{N,j}\right)\Big]\dif W_s^i.
\end{align*}
For simplicity of notation, $\tilde{b}(x,x')$ and $\tilde{\s}(x,x')$ can be redefined:
\beq\label{3-1}
\tilde{b}(x,x'):=b(x,x')-\int_{\mR^d}b(x,y)\mu_{x}(\dif y)
\deq
and
\beq\label{3-2}
\tilde{\s}(x,x'):=\s(x,x')-\int_{\mR^d}\s(x,y)\mu_{x}(\dif y).
\deq
By It\^o's formula and Jensen's inequality, we get
$$
\aligned
&\mE\Big[\sup_{t\in[0,T]}\Big|X_t^{N,i}-X_t^i\Big|^2\Big]
\\
\leq& 2\l_1\int_0^T\Big[\mE\Big(\sup_{r\in[0,s]}\Big|X_r^{N,i}-X_r^{i}\Big|^2\gamma_1\big(
\left|X_r^{N,i}-X_r^{i}\right|\big)\Big)\\
 &+\frac{1}{N}\sum_{j=1}^N\mE\Big(\sup_{r\in[0,s]}\left|X_r^{N,i}-X_r^{i}\right|\left|X_r^{N,j}-X
_r^{j}\right|\gamma_1\big(\left|X_r^{N,j}-X_r^{j}\right|\big)\Big)\Big]\dif s\\
&+\l_2\int_0^T\Big[\mE\Big(\sup_{r\in[0,s]}\left|X_r^{N,i}-X_r^{i}\right|^2\gamma_2\big(
\left|X_r^{N,i}-X_r^{i}\right|\big)\Big)\\
 &+
\frac{1}{N}\sum_{j=1}^N\mE\Big(\sup_{r\in[0,s]}\left|X_r^{N,j}-X_r^{j}\right|^2\gamma_2\left(
\left|X_r^{N,j}-X_r^{j}\right|\right)\Big)\Big]\dif s+\int_0^T\mE\Big[\sup_{r\in[0,s]}\left(\left|X_r^{N,i}-X_r^{i}\right|^2\right)\Big]\dif s\\
&+\int_0^T\mE\Big[\sup_{r\in[0,s]}\Big(\frac{1}{N}\sum_{j=1}^N\tilde{b}
\left(X_{r}^{i},X_{r}^{j}\right)\Big)^2 +\Big(\frac{1}{N}\sum_{j=1}^N\tilde{\s}\left(X_{r}^{i},X_{r}^{j}\right)\Big)^2\Big]\dif s.\\
\endaligned
$$
Using symmetry technique as in the proof of Lemma \ref{lem-2}, we have
\beq\label{2-4}
\aligned
&\mE\Big[\sup_{t\in[0,T]}\left|X_t^{N,i}-X_t^i\right|^2\Big]\\
\leq&4\l_1\int_0^T\mE\Big(\sup_{r\in[0,s]}\left|X_r^{N,i}-X_r^{i}\right|^2\gamma_1\left(
\left|X_r^{N,i}-X_r^{i}\right|\right)\Big)\dif s\\ &+ 2\l_2\int_0^T\mE\Big(\sup_{r\in[0,s]}\left|X_r^{N,i}-X_r^{i}\right|^2\gamma_2\left(
\left|X_r^{N,i}-X_r^{i}\right|\right)\Big)\dif s\\
& +\int_0^T \mE\Big(\sup_{r\in[0,s]}\left|X_r^{N,i}-X_r^{i}\right|^2\Big)\dif s
+\int_0^T \mE\Big[\sup_{r\in[0,s]}\Big|\frac{1}{N}\sum_{j=1}^N\tilde{b}\left(X_{r}^{i},X_{r}^{j}
\right)\Big|^2\Big]
\dif s\\
&
+\int_0^T\mE\Big[\sup_{r\in[0,s]}\Big|\frac{1}{N}\sum_{j=1}^N\tilde{\s}\left(X_{r}^{i},
X_{r}^{j}\right)\Big|^2
\Big]\dif s.\\
\endaligned
\deq

And due to the centering of $\tilde{b}(x,y)$ and $\tilde{\s}(x,y)$ with to their second variable, then if $j\neq k$ such that
\beq\label{3-3}
\mE\left[\tilde{b}(X_s^i,X_s^j)
\tilde{b}(X_s^i,X_s^k)\right]=0
\deq
and
\beq\label{3-4}
\mE\left[\tilde{\s}(X_s^i,X_s^j)
\tilde{\s}(X_s^i,X_s^k)\right]=0.
\deq
In light of Lemma \ref{lem-2}, we can obtain
\beq\label{3-5}
\aligned
&\mE\Big[\sup_{r\in[0,s]}\Big|\frac{1}{N}\sum_{j=1}^N\tilde{b}\left(X_r^{i},X_r^{j}\right)\Big|^2\Big]
=\frac{1}{N^2}\mE\Big[\sup_{r\in[0,s]}\Big|\sum_{j,k}\tilde{b}(X_r^i,X_r^j)
\tilde{b}(X_r^i,X_r^k)\Big|\Big]\\
&\qquad\qquad\qquad\leq  \frac{2}{N^2}\sum_{j=1}^N\mE\Big[\sup_{r\in[0,s]}\left|\tilde{b}\left(
X_r^{i},X_r^{j}\right)\right|^2\Big]\\
&\qquad\qquad\qquad\leq \frac{C}{N^2}\sum_{j=1}^N\Big(1+\mE\Big[\sup_{r\in[0,s]}|X_r^i|^2\Big]+\mE\Big[
\sup_{r\in[0,s]}|X_r^j|^2\Big]\Big) \leq \frac{C}{N}\\
\endaligned
\deq
and
\beq\label{3-6}
\mE\Big[\sup_{r\in[0,s]}\Big|\frac{1}{N}\sum_{j=1}^N\tilde{\s}\left(X_r^{i},X_r^{j}\right)\Big|^2\Big]
\leq \frac{C}{N}.
\deq

Furthermore, based on (\ref{a-2})-(\ref{a-4}) and (\ref{3-3})-(\ref{3-6}), the above inequality (\ref{2-4}) is
$$
\mE\Big[\sup_{t\in[0,T]}\left|X_t^{N,i}-X_t^i\right|^2\Big]\leq(2\l_1+\l_2)\int_0^T\rho_{\eta}
\Big(\mE\big(\sup_{r\in[0,s]}\left|X_r^{N,i}-X_r^i\right|^2\big)\Big)\dif s+\frac{C}{N}.
$$
Applying Gronwall-Belmman's inequality (Lemma \ref{lem-1}), we find
\beq\label{2-1}
\mE\Big[\sup_{t\in[0,T]}\left|X_t^{N,i}-X_t^i\right|^2\Big]\leq \frac{C}{N}.
\deq
 Thereby,
$$
\sup_{i=1,2,\cdots,N}N\mE\Big[\sup_{t\leq T}\left|X_t^{N,i}-X_t^i\right|^2\Big]<\infty.
$$
And by $(\ref{2-1})$, for $N$ is sufficiently large, the conclusion (\ref{1-4}) is established.
\end{proof}

\section{Proof of Theorem \ref{thm-2}}
\begin{lem}\label{lem-4} Under the assumptions of Theorem \ref{thm-2}, for any $T>0$ and some $\a\in(0,\frac{1}{2})$, there is a constant $C(T,\mE[|\xi|^2],\l_1,\l_2,c_0)>0$, we have
\beq\label{lem-4-1}
\sup_{i=1,2,\cdots,N}\mE\Big(\sup_{ t\in[0,T]}\big|X_t^{N,i}-X_t^{h,N,i}\big|^2\Big)<C h^{2\alpha}.
\deq
for $h\in(0,1)$ sufficiently small.
\end{lem}
\begin{proof}Set
$$
Z_t^{h,i}:=X_t^{N,i}-X_t^{h,N,i},
$$
then
$$
\aligned
Z_t^{h,i}=&\int_0^t\frac{1}{N}\sum_{j=1}^N\left[b\left(X_s^{N,i},X_s^{N,j}\right)-
b\left(X_{s_h}^{h,N,i},X_{s_h}^{h,N,j}\right)\dif s\right]\\ &+\int_0^t\frac{1}{N}\sum_{j=1}^N\left[\s\left(X_s^{N,i},X_s^{N,j}\right)-
\s\left(X_{s_h}^{h,N,i},X_{s_h}^{h,N,j}\right)\right]\dif W_s^i.
\endaligned
$$
Using It\^o's formula and taking the expectation, we find
$$
\aligned
&\mE\Big[\sup_{t\in[0,T]}\left|Z_t^{h,i}\right|^2\Big]=:Z(T)\\
\leq&\int_0^T\mE \Big[\sup_{r\in[0,s]}\left|Z_r^{h,i}\right|^2\left(2\l_1\gamma_1\left(\left|Z_r^{h,i}\right|\right)+
\l_2\gamma_2\left(\left|Z_r^{h,i}\right|\right)\right)\Big]\dif s\\
&+ \int_0^T\mE \Big\{ \sup_{r\in[0,s]} \frac{1}{N}\sum_{j=1}^N\left[2\l_1\left|Z_r^{h,i}\right|\left|Z_r^{h,j}\right|\gamma_1\left(\left|Z_r^{h,j}
\right|\right)+\l_2\left|Z_r^{h,j}\right|^2\gamma_2\left(\left|Z_r^{h,j}
\right|\right)\right]\Big\}\dif s\\
&+2\int_0^T\mE\Big[\sup_{r\in[0,s]} \frac{1}{N}\sum_{j=1}^N\left|Z_r^{h,i}\right|\left|b\left(X_r^{h,N,i},X_r^{h,N,j}\right)-
b\left(X_{r_h}^{h,N,i},X_{r_h}^{h,N,j}\right)\right|\Big]\dif s\\
&+\int_0^T\mE \Big[\sup_{r\in[0,s]}\frac{1}{N}\sum_{j=1}^N\left|\s\left(X_r^{h,N,i},X_r^{h,N,j}\right)-
\s\left(X_{r_h}^{h,N,i},X_{r_h}^{h,N,j}\right)\right|^2\Big]\dif s.
\endaligned
$$
With the aid of summing and symmetry in a similar way as Lemma \ref{lem-2}, we obtain
$$
\aligned
Z(T)
\leq &\int_0^T\mE \Big[\sup_{r\in[0,s]}\left|Z_r^{h,i}\right|^2\left(4\l_1\gamma_1\left(\left|Z_r^{h,i}\right|\right)+
2\l_2\gamma_2\left(\left|Z_r^{h,i}\right|\right)\right)\Big]\dif s+\int_0^T\mE\Big[\sup_{r\in[0,s]}\left|Z_r^{h,i}\right|^2\Big]\dif s\\
& +\int_0^T\mE \Big[ \sup_{r\in[0,s]}\frac{1}{N}\sum_{j=1}^N\left|b\left(X_r^{h,N,i},X_r^{h,N,j}\right)-
b\left(X_{r_h}^{h,N,i},X_{r_h}^{h,N,j}\right)\right|^2\Big]\dif s\\
&+\int_0^T\mE \Big[\sup_{r\in[0,s]}\frac{1}{N}\sum_{j=1}^N\left|\s\left(X_r^{h,N,i},X_r^{h,N,j}\right)-
\s\left(X_{r_h}^{h,N,i},X_{r_h}^{h,N,j}\right)\right|^2\Big]\dif s
\endaligned
$$
By lemma \ref{lem-3} and the condition $\mathbf{(H_2)}$ and Jensen's inequality, together with (\ref{a-2}) and (\ref{a-4}), there is a constant $\widetilde{C}=\widetilde{C}(\l_1,\l_2)$ independent of $h$, we have
$$
\aligned
Z(T)
\leq &C\int_0^T\rho_{\eta}\left(Z(s)\right)\dif s+[\rho_{\eta}(\widetilde{C}h^2)
+\rho^2_{\eta}(\widetilde{C}h)].
\endaligned
$$

And for  enough  small $x$, it is easy to see that $\rho_{\eta}^{1/2}(x)\leq x^{\alpha}(\alpha<\frac{1}{2})$. Using Lemma \ref{lem-1}, we obtain for some constant $C(T,\mE[|\xi|^2],\l_1,\l_2,c_0)>0$,
$$
\sup_{i=1,2,\cdots,N}\mE\Big(\sup_{ t\in[0,T]}\left|X_t^{N,i}-X_t^{h,N,i}\right|^2\Big)<Ch^{2\alpha},
$$
where $h$ is sufficiently small.
\end{proof}

\begin{proof}[Proof of Theorem \ref{thm-2}]
By Theorem \ref{thm-1} and Lemma \ref{lem-4}, the proof of Theorem \ref{thm-2} can be complete.
\end{proof}

{\bf Acknowledgement: } The authors would like to thank ZiMo Hao for quite useful suggestions about this paper.

\end{document}